\documentclass[preprint]{elsarticle}
\usepackage{amsmath,amsfonts}
\usepackage[utf8]{inputenc}
\newcommand{\rme}{\mathrm{e}}
\newcommand{\rmi}{\mathrm{i}}
\newcommand{\rmd}{\mathrm{d}}
\newcommand{\abs}[1]{\left\lvert#1\right\rvert}
\newcommand{\ZZ}{\mathbb{Z}}
\newcommand{\RR}{\mathbb{R}}
\newcommand{\CC}{\mathbb{C}}
\newcommand{\Acal}{\mathcal{A}}
\newcommand{\Bcal}{\mathcal{B}}
\newcommand{\Ccal}{\mathcal{C}}
\newcommand{\myre}{\text{Re\,}}
\newproof{proof}{Proof}
\newtheorem{thm}{Theorem}
\newdefinition{defn}{Definition}
\newdefinition{rmk}{Remark}
\newdefinition{hp}{Assumption}

\begin{document}
\title{A splitting approach\\ for the magnetic Schr\"odinger equation}
\author{M.~Caliari\corref{cor1}}
\ead{marco.caliari@univr.it}
\address{Dipartimento di Informatica, Universit\`a di Verona, Italy}
\author{A.~Ostermann\corref{}}
\ead{alexander.ostermann@uibk.ac.at}
\author{C.~Piazzola\corref{}}
\ead{chiara.piazzola@uibk.ac.at}
\address{Institut f\"ur Mathematik, Universit\"at Innsbruck, Austria}
\cortext[cor1]{Corresponding author}

\begin{abstract}
The Schr\"odinger equation in the presence of an external electromagnetic field is an important problem in computational quantum mechanics. It also provides a nice example of a differential equation whose flow can be split with benefit into three parts. After presenting a splitting approach for three operators with two of them being unbounded, we exemplarily prove first-order convergence of Lie splitting in this framework. The result is then applied to the magnetic Schr\"odinger equation, which is split into its potential, kinetic and advective parts. The latter requires special treatment in order not to lose the conservation properties of the scheme. We discuss several options. Numerical examples in one, two and three space dimensions show that the method of characteristics coupled with a nonequispaced fast Fourier transform (NFFT) provides a fast and reliable technique for achieving mass conservation at the discrete level.
\end{abstract}

\maketitle

\section{Introduction}
In quantum mechanics a lot of phenomena occur under the influence of an external electromagnetic field. Typical examples include the Zeeman effect, Landau levels and superconductivity. So, quite a few problems in computational solid state physics and quantum chemistry require the solution of the Schr\"{o}dinger equation in the presence of an electromagnetic field
\begin{equation}\label{original}
\begin{aligned}
\rmi\varepsilon \partial_t u &= \frac{1}{2} (\rmi\varepsilon\nabla+A)^2u+Vu,\quad t\ge 0, \ x \in \mathbb{R}^d, \\
u(0,x)&= u_0(x).
\end{aligned}
\end{equation}
Here, the unknown $u=u(t,x) \in \mathbb{C}$ is the quantum mechanical wave function, $V(t,x) \in \mathbb{R}$ is the scalar potential and $A(t,x) = (A_1(t,x),\dots, A_d(t,x))^{\sf T} \in \mathbb{R}^d$ is the vector potential. In addition  $\varepsilon \in (0,1]$ denotes the small semi-classical parameter which is the scaled Planck constant. The equation is considered subject to vanishing boundary conditions, i.e., $\lim_{\lvert x \rvert \to \infty}u(t,x)=0$. 
We recall that mass is a conserved quantity of this equation.

Exponential splitting schemes constitute a well-established class of methods for the numerical solution of Schr\"odinger equations (see, e.g., \cite{BJM02, F12, JL00, JMS11, L08}). In this approach, the kinetic part is solved in Fourier space, which gives spectral accuracy in space, whereas the multiplicative potential is integrated pointwise in physical space. The transformation between Fourier and physical space is carried out using the fast Fourier transform, which results in an overall fast algorithm. In our situation, however, when the vector potential depends on the position, we get an additional advection term, which cannot be handled efficiently with Fourier techniques.

Thus, the structure of problem \eqref{original} suggests to split the equation into three subproblems: a potential step which collects the scalar terms of the potentials (which are pointwise multiplications), a kinetic step which involves the Laplacian, and an advection step which results from the vector potential. For carrying out a time step, each of these steps is solved separately and their solutions are recombined to define the numerical approximation. This is the underlying idea of exponential splitting schemes (see \cite{S68, HLW00, MQ02}). In this paper we analyse a first-order method, the so-called Lie splitting. Note, however, that higher-order methods can be analysed in exactly the same way, if the underlying problem has enough spatial smoothness, see~\cite{HO09,JL00}.

Splitting the magnetic Schr\"odinger equation for the purpose of its numerical solution into three subproblems is not a new idea. In their recent paper \cite{JZ13}, Jin and Zhou proposed such a scheme. For the solution of the advection step, they considered a semi-Lagrangian approach. Such an approach has been used in many other fields as well (see, e.g., \cite{SRBG99, EO14, EO15}).

Our present paper differs from~\cite{JZ13} mainly in the following aspects: we give a framework for carrying out an abstract convergence proof for exponential splitting methods applied to \eqref{original}, and we give a detailed error analysis for the Lie splitting scheme by identifying the required smoothness assumptions on the data. Moreover, we address conservation properties of the scheme and identify an alternative to Lagrange interpolation, as the latter does not conserve mass.

The outline of this paper is as follows. We start in section \ref{Convergence} with an abstract convergence result for splitting into three subproblems. Guided by the properties of the magnetic Schr\"odinger equation, we present an analytic framework that allows us to prove convergence for exponential splitting schemes. We exemplify this by proving that Lie splitting applied to \eqref{original} has order of convergence one, as expected.

In section \ref{description} we apply a gauge transformation to the magnetic Schr\"odinger equation to obtain the equivalent formulation \eqref{main} with a divergence-free vector potential. This formulation is used in \eqref{eq:steps} to define the employed splitting. In the following section we show how to compute the solution of the kinetic step in spectral space and that of the potential step in physical space. For the advection step we use the method of characteristics. However, since the characteristic curves do not cut the previous time horizon at grid points, in general, special care has to be taken. We compare three different possibilities, namely discrete Fourier series evaluation, local polynomial interpolation and Fourier series evaluation by a nonequispaced fast Fourier transform (NFFT), see~\cite{KKP09}. The latter allows us to evaluate a Fourier series at an arbitrary set of points in a fast way. To our knowledge, this transform was not yet applied in the present context.

In section \ref{Numerical} we present some numerical results. Our main goal is the comparison of the different approximations used in the advection step. In particular, we study how well the considered numerical algorithms preserve mass, and how they compare in terms of computational efficiency.

\section{Splitting into three operators}\label{Convergence}
For the numerical solution of \eqref{original}, we propose a splitting approach. Motivated by the particular form of the vector field, which is the sum of a kinetic, a potential and an advective part, we consider a splitting into three terms. For this purpose, we formulate \eqref{original} as an abstract initial value problem
\begin{equation}\label{prb}
\begin{aligned}
\partial_t u &= (\Acal+\Bcal+\Ccal)u, \quad 0\le t \le T,\\
u(0)&=u_0
\end{aligned}
\end{equation}
in a Banach space $X$ with norm $\| \cdot \|$. Below, we will state an analytic framework for these operators $\Acal$, $\Bcal$ and $\Ccal$ that, on the one hand, is sufficiently general to include the magnetic Schr\"odinger equation as an example and, on the other hand, allows us to carry out an abstract convergence proof for (exponential) splitting methods. We will illustrate our approach by analysing in detail the Lie splitting scheme\footnote{Throughout the paper $\rme^{\tau\mathcal{L}}u_0$ will denote the exact solution at time $\tau$ of the abstract (linear) differential equation $\partial_t u = \mathcal{L}u$ with initial value $u(0)=u_0$.}
\begin{equation}\label{eq:lie}
u_{n+1}=\rme^{\tau \Ccal} \rme^{\tau \Acal} \rme^{\tau \Bcal} u_n,
\end{equation}
where $\tau$ denotes the step size and $u_n$ is the numerical approximation to the true solution $u(t)=\rme^{t(\Acal+\Bcal+\Ccal)} u(0)$ at time $t=t_n=n\tau$. We will show below that the Lie splitting scheme is first-order convergent. Let us stress, however, that exactly the same ideas can be used to analyse exponential splitting methods of higher order.

In a first step, we will study the local error $\|\rme^{\tau \Ccal} \rme^{\tau \Acal} \rme^{\tau \Bcal} u(t) - u(t+\tau)\|$ of Lie splitting along the exact solution. For this purpose, we employ the following assumption.

\begin{hp} \label{hp}
Let $\Bcal$ be a bounded operator, and let $\Acal$, $\Ccal$, and $\Acal+\Ccal$ generate strongly continuous semigroups $\rme^{t\Acal}$, $\rme^{t\Ccal}$, and $\rme^{t(\Acal+\Ccal)}$ on $X$. We assume that the following bounds hold for $0\le t\le T$ along the exact solution
\begin{subequations}
\begin{align}
\Vert [\Acal,\Ccal]\rme^{s\Acal} u(t) \Vert &\leq c_1,\label{hp1}\\
\Vert \Ccal \rme^{s\Acal} \Bcal u(t) \Vert &\leq c_2,\label{hp5}\\
\Vert \Ccal^2 \rme^{s\Acal} u(t) \Vert &\leq c_3,\label{hp2}\\
\Vert \Ccal\rme^{\sigma\Acal}\Ccal\rme^{s(\Acal+\Ccal)} u(t) \Vert &\leq c_4,\label{hp3}\\
\Vert [\Acal+\Ccal,\Bcal] \rme^{s(\Acal+\Ccal)} u(t)\Vert &\leq c_5\label{hp4}
\end{align}
\end{subequations}
with some constants $c_1$, $c_2$, $c_3$, $c_4$, and $c_5$ that do not depend on $0\le \sigma,s\le T$.
\end{hp}

Next, we recall the definition of the $\varphi_k$ functions, which play some role in our analysis. For complex $z$ and integer $k\ge 1$, we set
\begin{equation}\label{eq:phi}
\varphi_k(z) = \int_{0}^{1}\rme^{(1-\theta)z}\frac{\theta^{k-1}}{(k-1)!}\,\rmd\theta.
\end{equation}
These functions are uniformly bounded in the complex half-plane $\myre z \le 0$ and analytic in $\CC$. Let $\mathcal{E}$ be the generator of a strongly continuous semigroup. Then, for all $k \ge 1$, the following identity holds in the domain of $\mathcal{E}^k$
\begin{equation}\label{eq:taylor-phi}
\rme^{\tau \mathcal{E}} = \sum_{j=0}^{k-1} \frac{\tau^j}{j!}\mathcal{E}^j+\tau^k\mathcal{E}^k\varphi_k(\tau \mathcal{E}).
\end{equation}
We are now in the position to state the local error bound.
\begin{thm}[Local error bound]
Under Assumption \ref{hp}, the following bound for the local error holds
\begin{equation} \label{err}
\Vert \rme^{\tau \Ccal} \rme^{\tau \Acal} \rme^{\tau \Bcal} u(t)-u(t+\tau) \Vert \leq C \tau^2, \qquad t \in [0,T-\tau]
\end{equation}
with a constant $C$ that does not depend on $t$ and $\tau$.
\end{thm}

\begin{proof}
Our proof uses ideas developed in \cite{JL00}. Since $\Bcal$ is bounded, the numerical solution can be expanded in the following way
\begin{equation} \label{LocErr}
\begin{aligned}
\rme^{\tau \Ccal} \rme^{\tau \Acal} \rme^{\tau \Bcal} u(t)
& = \rme^{\tau \Ccal} \rme^{\tau \Acal}\left( I + \tau \Bcal+ \mathcal{O}(\tau^2)\right)u(t)\\
& = \underbrace{\rme^{\tau \Ccal} \rme^{\tau \Acal}u(t)}_{P_1}+\underbrace{\tau \rme^{\tau \Ccal} \rme^{\tau \Acal} \Bcal u(t)}_{Q_1} + \mathcal{O}(\tau^2).
\end{aligned}
\end{equation}
The exact solution, on the other hand, is expanded with the help of the variation-of-constants formula. Applying this formula twice yields the representation
\begin{equation} \label{exact}
\begin{aligned}
\rme^{\tau (\Acal+\Bcal+\Ccal)}u(t)
&= \rme^{\tau(\Acal+\Ccal)}u(t)+\int_0^{\tau}\rme^{s(\Acal+\Ccal)}\Bcal\rme^{(\tau-s) (\Acal+\Bcal+\Ccal)}u(t)\rmd s\\
& = \underbrace{\rme^{\tau(\Acal+\Ccal)}u(t)}_{P_2} +\underbrace{\int_0^{\tau}\rme^{s(\Acal+\Ccal)}\Bcal\rme^{(\tau-s) (\Acal+\Ccal)}u(t)\rmd s}_{Q_2}+ \mathcal{O}(\tau^2).
\end{aligned}
\end{equation}
Collecting all the terms, we can rewrite the local error as
\begin{equation}\label{eq:le}
\rme^{\tau \Ccal} \rme^{\tau \Acal} \rme^{\tau \Bcal} u(t)- u(t+\tau) = P+Q+ \mathcal{O}(\tau^2),
\end{equation}
where $P = P_1-P_2$ and $Q = Q_1-Q_2$.

For expanding $P_1$ we employ the $\varphi_2$ function (see \eqref{eq:taylor-phi}) to get
\[
\rme^{\tau \Ccal} \rme^{\tau \Acal}u(t) = \rme^{\tau \Acal}u(t)+ \tau \Ccal \rme^{\tau \Acal}u(t) +\tau^2 \Ccal^2 \varphi_2(\tau\Ccal)\rme^{\tau \Acal}u(t).
\]
Using the variation-of-constants formula twice, we can rewrite $P_2$ as
\[
\begin{aligned}
\rme^{\tau(\Acal+\Ccal)}u(t) &= \rme^{\tau\Acal}u(t)+\int_0^{\tau}\rme^{s\Acal}\Ccal\rme^{(\tau-s) \Acal}u(t)\,\rmd s\\
&\quad +\int_0^{\tau}\rme^{s\Acal}\Ccal\int_0^{\tau-s}\rme^{\sigma\Acal}\Ccal\rme^{(\tau-s-\sigma) (\Acal+\Ccal)}u(t) \,\rmd \sigma \rmd s.
\end{aligned}
\]
Thus, to bound $P$, we need first to estimate
\begin{equation} \label{1P}
\tau \Ccal \rme^{\tau \Acal}u(t)-\int_0^{\tau}\rme^{s\Acal}\Ccal\rme^{(\tau-s)\Acal}u(t)\rmd s,
\end{equation}
and then to bound the remaining terms.
Let $f(s)=\rme^{s\Acal}\Ccal\rme^{(\tau-s) \Acal}u(t)$. Then, the expression \eqref{1P} becomes
\[
\tau f(0)-\int_0^{\tau} f(s) \rmd s= \tau f(0)- \int_0^{\tau} \left(f(0) + \int_0^s f'(\sigma) \rmd \sigma \right) \rmd s = -\int_0^{\tau}\!\!\! \int_0^s f'(\sigma) \rmd \sigma \rmd s,
\]
and can be bounded with assumption \eqref{hp1}
\begin{equation} \label{T1}
\bigg\Vert \int_0^{\tau}\!\!\! \int_0^s \rme^{\sigma\Acal}[\Acal,\Ccal]\rme^{(\tau-\sigma)\Acal} u(t) \rmd \sigma \rmd s \bigg\Vert \leq c\tau^2.
\end{equation}
Furthermore, by employing assumptions \eqref{hp2} and \eqref{hp3}, the remaining terms in $P$ can be estimated as
\begin{equation} \label{T2}
\lVert \tau^2 \Ccal^2 \varphi_2(\tau\Ccal)\rme^{\tau \Acal}u(t) \Vert = \Vert \tau^2 \varphi_2(\tau\Ccal)\Ccal^2 \rme^{\tau \Acal}u(t) \rVert  \leq c\tau^2
\end{equation}
and
\begin{equation} \label{T3}
\bigg\lVert \int_0^{\tau}\rme^{s\Acal}\Ccal\int_0^{\tau-s}\rme^{\sigma\Acal}\Ccal\rme^{(\tau-s-\sigma) (\Acal+\Ccal)}u(t) \rmd \sigma \rmd s \bigg\rVert \leq c\tau^2.
\end{equation}
Taking all together, we have shown that $P=\mathcal{O}(\tau^2)$.

As regards $Q$, by setting $g(s)=\rme^{s(\Acal+\Ccal)}\Bcal\rme^{(\tau-s) (\Acal+\Ccal)}u(t)$ and proceeding in the same way as for \eqref{1P} we obtain
\[
\begin{aligned}
Q &= \tau \rme^{\tau \Ccal} \rme^{\tau \Acal} \Bcal u(t) - \int_{0}^{\tau} g(s) \,\rmd s\\
&=\tau \rme^{\tau \Ccal} \rme^{\tau \Acal} \Bcal u(t) - \tau g(\tau) - \int_{0}^{\tau} \!\!\!\int_{\tau}^{s}g'(\sigma) \,\rmd \sigma \rmd s\\
&=\tau \rme^{\tau \Ccal} \rme^{\tau \Acal} \Bcal u(t)-\tau \rme^{\tau(\Acal+\Ccal)}\Bcal u(t) \\
&\quad - \int_0^{\tau}\!\!\!\int_{\tau}^{s} \rme^{\sigma(\Acal+\Ccal)}[\Acal+\Ccal,\Bcal] \rme^{(\tau-\sigma)(\Acal+\Ccal)}u(t)\,\rmd \sigma \rmd s.
\end{aligned}
\]
The double integral is bounded with the help of assumption \eqref{hp4} by $c\tau^2$. For the remaining two terms, we use that
$$
\rme^{\tau \Ccal} \rme^{\tau \Acal} \Bcal u(t) = \rme^{\tau \Acal} \Bcal u(t) + \tau \Ccal\varphi_1(\tau \Ccal) \rme^{\tau \Acal} \Bcal u(t)
$$
and employ once more the variation-of-constants formula
$$
\rme^{\tau(\Acal+\Ccal)}\Bcal u(t) = \rme^{\tau \Acal} \Bcal u(t) + \int_0^{\tau} \rme^{s\Acal}\Ccal\rme^{(\tau-s)(\Acal+\Ccal)}\Bcal u(t)\,\rmd s.
$$
Assumption \eqref{hp5} shows that their difference is again bounded by $c\tau^2$. From this we conclude the assertion.\qed
\end{proof}

Assumption~\ref{hp} guarantees that the semigroups, generated by $\Acal$, $\Bcal$, and $\Ccal$ satisfy the bounds
\[
\Vert \rme^{t \Acal}\Vert \leq M_1 \rme^{t\omega_1}, \quad
\Vert \rme^{t \Bcal}\Vert \leq \rme^{t\omega_2}, \quad
\Vert \rme^{t \Ccal}\Vert \leq M_3\rme^{t\omega_3},\qquad t\ge 0
\]
for some constants $M_1\ge 1$, $M_3\ge 1$, $\omega_1$, $\omega_2$, and $\omega_3$. Moreover, it is possible to choose an equivalent norm $\|\cdot\|_*$ on $X$ such that $\Vert \rme^{t \Acal}\Vert_* \leq \rme^{t\omega_{\Acal}}$. Unfortunately, this is still not enough to prove stability, in general. Therefore, we impose an additional assumption.

\begin{hp} \label{hpStab}
There is a constant $\omega_C$ such that $\Vert \rme^{t \Ccal}\Vert_* \leq \rme^{t\omega_{\Ccal}}$ for all $t\ge 0$.
\end{hp}
Under this additional assumption, it is easy to show stability.

\begin{thm} [Stability]
Under Assumptions \ref{hp} and \ref{hpStab}, Lie splitting is stable, i.e., there is a constant $C$ such that
\begin{equation} \label{stab}
\left\Vert \left(\rme^{\tau \Ccal} \rme^{\tau \Acal} \rme^{\tau \Bcal} \right)^j\right\Vert \leq C
\end{equation} 	
for all $j\in \mathbb{N}$ and $\tau\ge 0$ satisfying $0\le j\tau \le T$.\qed
\end{thm}

\begin{proof}
Our assumptions imply that $\left\|\rme^{\tau \Ccal}\rme^{\tau \Acal}\rme^{\tau \Bcal}\right\|_* \le \rme^{\tau(\omega_{\Acal} + \omega_{\Bcal} + \omega_{\Ccal})}$ from which the assertion follows.
\qed
\end{proof}

From consistency and stability, convergence follows in a standard way.

\begin{thm}[Global error bound]\label{thm:convergence}
Under Assumptions \ref{hp} and \ref{hpStab}, the Lie splitting discretization \eqref{eq:lie} of the initial value problem \eqref{prb} is convergent of order~1, i.e., there exists a constant $C$ such that
$$
\Vert u_n-u(t_n) \Vert \leq C \tau,
$$
for all $n\in\mathbb{N}$ and $\tau> 0$ satisfying $0\le n\tau = t_n \le T$.
\end{thm}

\begin{proof}
We express the global error with the help of a telescopic sum
\[
\begin{aligned}
u_n-u(t_n)& =\left(\left(\rme^{\tau \Ccal} \rme^{\tau \Acal} \rme^{\tau \Bcal} \right)^n -\rme^{n\tau(\Acal+\Bcal+\Ccal)}\right)u(0) \\
& =\sum_{j=0}^{n-1} \left(\rme^{\tau \Ccal} \rme^{\tau \Acal} \rme^{\tau \Bcal} \right)^{n-j-1} \left(\rme^{\tau \Ccal} \rme^{\tau \Acal} \rme^{\tau \Bcal}u(t_j)-u(t_{j+1})\right)
\end{aligned}
\]
and use the estimates \eqref{err} and \eqref{stab}.\qed
\end{proof}

\section{Example: the magnetic Schr\"{o}dinger equation} \label{description}
The electromagnetic field in $\RR^3$ is the combination of the electric field $E$ and the magnetic field $B$. Both fields depend on time and space, in general. Mathematically, they are given by a scalar potential $V$ and a vector potential $A$, respectively
\[
E =- \nabla V-\frac{\partial A}{\partial t}, \qquad B= \nabla \times A.
\]
Making use of the fact that we can impose conditions on the potentials as long as we do not affect the resulting fields, we will apply the following transformations
\begin{equation} \label{transform}
\begin{aligned}
\tilde{u}(t,x) & = u(t,x) \,\rme^{\rmi\lambda(t,x)}, \\
\tilde{A}(t,x) & = A(t,x)+\varepsilon\nabla\lambda(t,x), \\
\tilde{V}(t,x) & = V(t,x)-\varepsilon\partial_t \lambda(t,x).
\end{aligned}
\end{equation}
One natural choice is to impose a so-called Coulomb gauge, i.e., to select $\lambda$ in such a way that $\nabla \cdot \tilde{A}=0$. Consequently, this gauge $\lambda$ has to satisfy the Poisson equation $\varepsilon \Delta \lambda = -\nabla\cdot A$.

Applying \eqref{transform} to the Schr\"odinger equation \eqref{original} and dropping right away the tildes, we obtain the following problem
\begin{equation} \label{main}
\begin{aligned}
\rmi\varepsilon \partial_t u &= -\frac{\varepsilon^2}{2}\Delta u+\rmi\varepsilon A \cdot \nabla u+ \frac{1}{2}\lvert A \rvert ^2 u +Vu, \quad t\in [0,T],\\
u(0,x)&=u_0(x)
\end{aligned}
\end{equation}
with a divergence-free vector potential $A$.

We are now in the position to give a precise formulation of the three subproblems that are used in our splitting. Henceforth, they will be called potential, kinetic and advection step, respectively:
\begin{subequations}\label{eq:steps}
\begin{align}
\partial_t u & =\Bcal u = -\frac{\rmi}{\varepsilon}\left(\frac12\lvert A \rvert^2 +V\right)u, \label{potential}\\
\partial_t u &= \Acal u = \frac{\rmi\varepsilon}{2} \Delta u,  \label{kinetic}\\
\partial_t u &= \Ccal u= A \cdot \nabla u,\qquad \nabla\cdot A=0.  \label{convection}
\end{align}
\end{subequations}
The kinetic step \eqref{kinetic} can be handled analytically in Fourier space, whereas the potential step \eqref{potential} is easily performed in physical space. For the advection step \eqref{convection} we will present three modifications of a semi-Lagrangian method in section \ref{sec:adv} below.

An important feature of \eqref{original} and \eqref{main} is the conservation of mass $m = \Vert u (t,\cdot)\Vert_{L^2}^2$, i.e., $\frac{\partial}{\partial t} \Vert u(t,\cdot) \Vert_{L^2}^2 = 0$.
The split step solution based on \eqref{eq:steps} is mass conserving as well. Indeed, the kinetic step preserves the $L^2$ norm due to Parseval's identity. The modulus of the solution of the potential step is preserved, and we are also able to show that the advection step conserves the mass. This is seen by
multiplying \eqref{convection} by $\overline u$
\[
\overline{u}\,\partial_t u-\overline{u}\, A\cdot \nabla u= 0
\]
and adding this equation to its complex conjugate, which results in
\[
\partial_t \lvert u \rvert^2 = A \cdot \nabla \lvert  u \rvert^2.
\]
Integrating this last equation by parts shows
\[
\partial_t\Vert u \Vert_{L^2}^2 = \int \partial_t \lvert u \rvert^2 \,\rmd x = \int A \cdot \nabla \lvert u \rvert ^2  \,\rmd x = -\int \lvert u \rvert ^2 \nabla \cdot A \,\rmd x =0,
\]
where the last identity follows from the Coulomb gauge.

Henceforth, we consider \eqref{main} and \eqref{eq:steps} on the hyperrectangle $\Omega=\Pi_{i=1}^d[a_i,b_i)$, subject to periodic boundary conditions. In particular, the potentials $V$ and $A$ are assumed to be periodic functions on $\Omega$. Then, all what has been said in this section remains valid.

We finally remark that our splitting approach also works if $A$ is not divergence-free. In this case the potential operator is given by
\[
\mathcal{B}=-\frac{\rmi}{\varepsilon}\left(\frac{1}{2}\lvert A \rvert^2+V\right)+\frac{1}{2}\nabla \cdot A,
\]
whereas the other two operators stay the same. However, in this case, we will lose the conservation of mass for the potential step.

\section{Space discretization, potential and kinetic step}
We discretize the hyperrectangle $\Omega=\prod_{i=1}^d [a_i,b_i)$ by a regular grid. For $1\le i \le d$, let $N_i\ge 2$ be an even integer and let
\begin{equation}\label{eq:indexset}
I_N=\ZZ^d\cap \prod_{i=1}^d\left[-\tfrac{N_i}{2},\tfrac{N_i}{2}\right).
\end{equation}
For $j=(j_1,\ldots,j_d)\in I_N$ we consider the grid points $x^j$ with components
$$
x^j_i = \frac{a_i+b_i}2 +\frac{j_i}{N_i}(b_i-a_i),\qquad 1\le i\le d.
$$
For performing the potential step, we solve the ordinary differential equation~\eqref{potential} at each grid point $x^j$. More precisely, starting with an initial value $v$ at time $t_n$, we solve
$$
\dot w(s) = -\frac{\rmi}{\varepsilon}\left(\frac12\left| A(t_n+s,x^j) \right|^2 +V(t_n+s,x^j)\right)w(s), \quad w(0) = v(x^j)
$$
to obtain
$$
\left(\rme^{\tau\cal B}v\right)(x^j) = w(\tau).
$$
If the potentials $A$ and $V$ are time-independent, the analytic solution is readily available. Otherwise, a quadrature method (up to machine precision) can be employed.

The kinetic step is approximated in Fourier space. For a given function
\begin{equation*}
v\colon \prod_{i=1}^d [a_i,b_i)\to \CC,
\end{equation*}
let $\hat v_k$ denote its Fourier coefficients, i.e.
$$
v(x)=\sum_{k\in I_N} \hat v_k E_k(x), \qquad
E_k(x)=\prod_{i=1}^d\frac{\rme^{2\pi\rmi k_i(x_i-a_i)/(b_i-a_i)}}{\sqrt{b_i-a_i}},
$$
where $x=(x_1,\ldots,x_d)$. Further, let
$$
\lambda_i = \frac{\rmi \varepsilon}2\left(\frac{2\pi k_i}{b_i-a_i}\right)^2.
$$
The Fourier coefficients of $\rme^{\tau\cal A}v$ are then given by $\rme^{\tau\lambda_i} \hat v_{k_i}$. The transformation between physical and Fourier space is usually carried out with the fast Fourier transform.

\section{Advection step}\label{sec:adv}
In this section we describe the solution of the advection step
\begin{equation}\label{eq:adv}
\left\{
\begin{aligned}
\partial_t v(t,x)&=A(x)\cdot \nabla v(t,x),&t&\in[0,\tau],\\
v(0,x)&=v_0(x).
\end{aligned}\right.
\end{equation}
Our approach is based on the method of characteristics, i.e., we make use of the curves $s\mapsto x(s)\in \RR^d$ satisfying the $d$-dimensional system of ordinary differential equations
\begin{equation*}
\dot x(s)=-A(x(s)).
\end{equation*}
Since the solution of the advection equation \eqref{eq:adv} is constant along characteristics, we have $v(\tau,x^j)=v(0,x^j(0))$ for each grid point $x^j$, $j\in I_N$, where $x^j(0)$ denotes the solution of
\begin{equation}\label{eq:ODE}
\left\{
\begin{aligned}
\dot x^j(s)&=-A(x^j(s)),&s&\in[0,\tau],\\
x^j(\tau)&=x^j
\end{aligned}
\right.
\end{equation}
at $s=0$. This system can be solved once and for all for each grid point with an explicit method at high precision, if the time step $\tau$ is kept constant. However, since $x^j(0)$ is not a grid point, in general, the value $v(0,x^j(0))=v_0(x^j(0))$ has to be recovered. We describe here three different procedures for achieving the evaluation of $v_0(x^{j}(0))$ at the set of $\prod_iN_i$ points $\{x^{j}(0)\}_j$. For the sake of simplicity, we only describe the one-dimensional case in detail. However, we also report the overall computational cost for the general $d$-dimensional case.

We remark that the same approach can be used for time dependent potentials $A(t,x)$. Instead of \eqref{eq:ODE} one has to solve the non-autonomous problem
\begin{equation}
\left\{
\begin{aligned}
\dot x^j(s)&=-A(t_n+s, x^j(s)),&s&\in[0,\tau],\\
x^j(\tau)&=x^j.
\end{aligned}
\right.
\end{equation}
Its numerical solution at $s=0$ is again used to define the sought-after approximation $v(\tau,x^j) = v(0,x^j(0))$.

\subsection{Direct Fourier series evaluation}
Since the initial value $v_0(x)$ of the advection step is the result of the solution of the kinetic step, the function $v_0$ is known through its Fourier coefficients $\{\hat v_k\}_k$. It is
therefore possible to directly evaluate
\begin{equation}\label{eq:DFS}
v_0(x^j(0))=
\sum_{k\in I_N}\hat v_{k}E_k(x^j(0)).
\end{equation}
In the $d$-dimensional case, the $\prod_i N_i^2$ values $E_k(x^j(0))$ can be precomputed once and for all, if the time step $\tau$ is constant. The evaluation cost of~\eqref{eq:DFS} at the point set $\{x^j(0)\}_j$ is then $\mathcal{O}(\prod_i N_i^2)$ at each time step.

\subsection{Local polynomial interpolation}
Another possibility (see, for instance, \cite{SRBG99,JZ13,EO15}) is local polynomial interpolation. It is possible to evaluate $v_0(x)$ at the grid points $\{x^j\}_j$ with an inverse fast Fourier transform of cost $\mathcal{O}\left(N_1\cdot\ldots\cdot N_d\cdot(\log N_1+\ldots+\log N_d)\right)$. An approximation of the values $v_0(x^j(0))$ can then be obtained by local polynomial interpolation
\begin{equation}\label{eq:PE}
v_0(x^j(0))\approx \sum_{k\in I_p}v_0(x^{j+k})L_{j+k}(x^j(0)).
\end{equation}
Here $\{x^{j+k}\}_k$ is the set of the $p$ grid points, $p$ even, satisfying
\begin{equation*}
x^{j-p/2}<\ldots<x^{j-1}\le x^j(0)<x^j<\ldots<x^{j+p/2-1},
\end{equation*}
and $L_{j+k}$ denotes the elementary Lagrange polynomial of degree $p-1$ that takes the value one at $x^{j+k}$ and zero at all the other $p-1$ points. Of course, the points $x^{j+k}$ and the corresponding values $v_0(x^{j+k})$ have to be taken by periodicity if necessary.

In the $d$-dimensional case, for a constant time step $\tau$ it is possible to precompute once and for all the elementary Lagrange polynomials at the points $\{x^j(0)\}_j$ (for a total amount of $p^d\prod_i N_i$ values). Then, the evaluation of \eqref{eq:PE} at each time step requires $\mathcal{O}(p^d\prod_i N_i)$ operations.

\subsection{Fourier series evaluation by NFFT}
The third explored possibility is the evaluation of~\eqref{eq:DFS} by means of an approximate fast Fourier transform. Among others, we tested the nonequi\-spaced fast Fourier transform (NFFT) by Keiner, Kunis and Potts~\cite{KKP09}. The computational cost of such an approach is $\mathcal{O}\bigl(N_1\cdot\ldots\cdot N_d\cdot(\log N_1+\ldots+\log N_d+ \abs{\log\epsilon}^d)\bigr)$, where $\epsilon$ is the desired accuracy.

For the readers' convenience, we briefly sketch the NFFT algorithm in one dimension, using the original notation of~\cite{KKP09}. Given some coefficients $\{\hat f_k\}_{k\in I_N}$, $N$ even, and a set of \emph{arbitrary} points $\{x^j\}_j\subset\left[-\frac12,\frac12\right)$, the aim is a fast evaluation of the one-periodic trigonometric polynomial
\begin{equation}\label{eq:f}
f(x)=\sum_{k\in I_N}\hat f_k \rme^{-2\pi\rmi kx}
\end{equation}
at the points $\{x^j\}_j$. In the first step, $f(x)$ is replaced with the ansatz
\begin{equation*}
s_1(x)=\sum_{\ell\in I_{n}}g_\ell\,\tilde\varphi\left(x-\tfrac{\ell}{n}\right),\quad \sigma\ge 2,\quad \text{$n=\sigma N$ even},
\end{equation*}
where $\{g_\ell\}_\ell$ are some coefficients to be defined later and
\begin{equation*}
\tilde \varphi(x)=\sum_{r\in\ZZ}\varphi(x+r)
\end{equation*}
is the one-periodic version of a \emph{window} function $\varphi\colon\RR\to\RR$. The window function $\varphi$ is chosen in such a way that $\tilde \varphi$ has a uniformly convergent Fourier series
\begin{equation*}
\tilde\varphi(x)=\sum_{k\in\ZZ}c_k(\tilde \varphi)\rme^{-2\pi\rmi kx}.
\end{equation*}
The default window function used by NFFT is the so called \emph{Keiser--Bessel function}
\begin{equation*}
\varphi(x)=\frac{1}{\pi}\left\{
\begin{aligned}
&\frac{\sinh(\beta\sqrt{m^2-n^2x^2})}{\sqrt{m^2-n^2x^2}}&&\text{for $\abs{x}<\frac{m}{n}$},\\
&\frac{\sin(\beta\sqrt{n^2x^2-m^2})}{\sqrt{n^2x^2-m^2}}&&\text{for $\abs{x}>\frac{m}{n}$},\\
&\beta&&\text{for $\abs{x}=\frac{m}{n}$}
\end{aligned}
\right.
\end{equation*}
with the \emph{shape} parameter $\beta=\pi(2-1/\sigma)$. The value of $m$ depends on the desired accuracy $\epsilon$ and is chosen $m=8$ for double precision. The \emph{oversampling} factor $\sigma$ is defined by
\begin{equation*}
\sigma=\frac{2^{\lceil\log_2 2N\rceil}}{N}.
\end{equation*}
That is, $n=\sigma N$ is the smallest power of two with $2\le \sigma<4$. Now we plug the Fourier series expansion of $\tilde\varphi(x)$ into $s_1(x)$ in order to get
\begin{equation*}
\begin{aligned}
s_1(x)&=\sum_{\ell\in I_{n}}g_\ell\,\tilde \varphi\left(x-\tfrac{\ell}{n}\right)\\
&= \sum_{\ell\in I_{n}}g_\ell\sum_{k\in\ZZ}c_k(\tilde\varphi)\rme^{-2\pi\rmi k\left(x-\frac{\ell}{n}\right)}\\
&=\sum_{k\in\ZZ}\left(\sum_{\ell\in I_{n}}g_\ell\rme^{2\pi\rmi k\frac{\ell}{n}}\right)c_k(\tilde\varphi)\rme^{-2\pi\rmi kx}
\end{aligned}
\end{equation*}
and apply a cutoff in the frequency domain
\begin{equation}\label{eq:s1cutoff}
s_1(x)\approx \sum_{k\in I_{n}}\left(\sum_{\ell\in I_{n}}g_\ell\rme^{2\pi\rmi k\frac{\ell}{n}}\right)c_k(\tilde\varphi)\rme^{-2\pi\rmi kx}
=\sum_{k\in I_{n}}\hat g_kc_k(\tilde\varphi)\rme^{-2\pi\rmi k x}.
\end{equation}
Comparing now equations~\eqref{eq:f} and \eqref{eq:s1cutoff}, we see that the coefficients $\{\hat g_k\}_k$ are simply given by
\begin{equation*}
\hat g_k=\left\{
\begin{aligned}
&\frac{\hat f_k}{c_k(\tilde \varphi)},&&k\in I_N,\\
&0,&&k\in I_n\setminus I_N,
\end{aligned}
\right.
\end{equation*}
and the values $\{g_\ell\}_\ell$ can be recovered by a fast Fourier transform of length $n$. The parameter $m$ is then used as a \emph{cutoff} to approximate in practice the window function $\varphi(x)$ with
\begin{equation*}
\psi(x)=\varphi(x)\chi_{\left[-\frac{m}{n},\frac{m}{n}\right]}(x).
\end{equation*}
In this way, $s_1(x)$ is further approximated by
\begin{equation*}
s_1(x)\approx s(x)=\sum_{\ell\in I_{n}}g_\ell\,\tilde \psi\left(x-\tfrac{\ell}{n}\right).
\end{equation*}
Now we use that $\tilde \psi$ vanishes outside of $-\frac{m}{n}\le x-\frac{\ell}{n}\le \frac{m}{n}$. Thus, for fixed $x^j$, the above sum contains at most $2m+1$ terms different from zero. Finally, $s(x)$ is evaluated at the set $\{x^j\}_j$, providing the desired approximation of $\{f(x^j)\}_j$.

\section{Application to the magnetic Schr\"odinger equation} \label{Verification}
In this section we exemplify the assumptions of Theorem~\ref{thm:convergence} for the magnetic Schr\"{o}dinger equation \eqref{main}. For this purpose, we choose $X=L^2(\Omega)$ with $\Omega = \prod_{i=1}^d [a_i,b_i)$ and assume that the potentials $A$ and $V$ are sufficiently smooth. Note that the potential operator $\Bcal$ is bounded, whereas the kinetic operator $\Acal$ and the advection operator $\Ccal$ are both unbounded. We start with the verification of Assumption \ref{hp}.
\begin{itemize}
\item[$\diamond$]
Condition \eqref{hp1}: Since $\rme^{t\Acal}w$ is the exact solution of the problem $\partial_t u=\Acal u$, $u(0)=w$, it preserves the smoothness of the initial data. Further, the commutator $[\Acal,\Ccal]$ is a second-order differential operator
\begin{align*}
[\Acal,\Ccal] u &= \frac{\rmi \varepsilon}{2}[\Delta, A\cdot \nabla]u\\
&= \frac{\rmi\varepsilon}{2} \Bigl(\Delta(A\cdot \nabla u)-A\cdot \nabla(\Delta u) \Bigr).
\end{align*}
So, we need to assume that the initial data are twice differentiable.
\item[$\diamond$]
Conditions \eqref{hp5}, \eqref{hp2}, and \eqref{hp3}: As $\Ccal$ is a first-order differential operator, it is again sufficient to require that the initial data are twice differentiable.
\item[$\diamond$]
Condition \eqref{hp4}: The commutator is a second-order differential operator
\[
[\Acal+\Ccal,\Bcal]u = \left[\frac{\rmi\varepsilon}{2} \Delta + A \cdot \nabla, -\frac{\rmi}{\varepsilon}\left(\frac{1}{2}\lvert A \rvert^2+V\right)\right]u,
\]
so the same smoothness as before is required.
\end{itemize}

Stability is easily verified. From the conservation of mass discussed at the end of section \ref{description}, we get $\lVert \rme^{\tau\Acal} \rVert_{L^2} = 1$ and $\lVert \rme^{\tau\Ccal} \rVert_{L^2} = 1$.

Note that the above bound for the advection semigroup only holds in the Coulomb gauge setting. However, by the method of characteristics, the solution of the advection step is of the form $u(t,x(t))=u_0(x(0))$, where $x(t)=x(0)+tA(x(0))+\mathcal{O}(t^2).$ Setting $\xi=x(0)$, we have
\[
\Vert u \Vert_{L^2}^2= \int_{\Omega} \vert u(x) \vert^2 \rmd x=  \int_{\Omega} \vert u_0(\xi) \vert^2 \rmd x = \int_{\Omega} \vert u_0(\xi) \vert^2 \left|\det\bigl(I+tA'(\xi)+\mathcal O(t^2)\bigr)\right|\rmd\xi.
\]	
Under the assumption that the partial derivatives of $A$ are bounded, we have
\[
\Vert u \Vert_{L^2}^2 \leq \Vert u_0 \Vert^2_{L^2}+ C t \Vert u_0 \Vert^2_{L^2}\leq (1+Ct)\Vert u_0 \Vert^2_{L^2} \le \rme^{2t\omega_{\Ccal}}\Vert u_0 \Vert^2_{L^2},
\]
which is exactly the weaker bound required in Assumption \ref{hpStab}.

\section{Numerical experiments} \label{Numerical}
The first numerical example is a variation of~\cite[Example 2]{JZ13}. The vector potential is chosen as $A(x)=\sin(2\pi x)/5+1/5$ and the scalar potential as $V(x)=\cos(2\pi x)/5+4/5$. The initial value is $u_0(x)=\sqrt{\rho_0(x)}\exp(\rmi S_0(x)/\varepsilon)$, where
\begin{equation*}
\rho_0(x)=\rme^{-50\left(x-\frac{1}{2}\right)^2},\quad
S_0(x)=-\frac{\log\bigl(\rme^{5(x-\frac{1}{2})}+\rme^{-5(x-\frac{1}{2})}\bigr)}{5},\quad \varepsilon=\frac{1}{128}.
\end{equation*}
Note that this initial value is not periodic. However, due to the exponential decay of $\rho_0(x)$, the problem can be solved numerically up to time $T=0.42$ in the space interval $[0,1]$ by assuming periodic boundary conditions. The Coulomb gauge transformation yields
\begin{equation*}
\lambda(x)=\frac{\cos(2\pi x)}{10\pi\varepsilon}.
\end{equation*}
\begin{figure}[tb]
\centering
\includegraphics[scale=0.45]{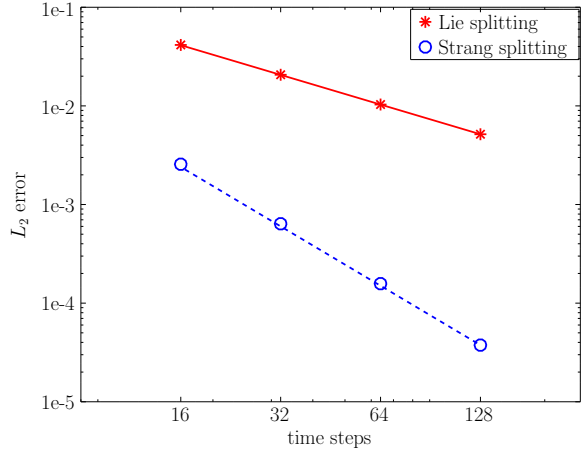}
\caption{Temporal errors (stars, circles) for the Lie and Strang splitting
methods and reference orders (lines) for the one-dimensional example.}
\label{fig:order1d}
\end{figure}%
In Figure~\ref{fig:order1d} we plot the global errors of Lie splitting at the final time $T=0.42$ for various time steps and $N=2048$ spatial discretization points. The reference solution was obtained with 512 time steps. We include in this figure the error behavior of Strang splitting, defined by
\begin{equation}\label{eq:strang}
u_{n+1}=\rme^{\frac{\tau}{2} \Bcal}\rme^{\frac{\tau}{2} \Acal}\rme^{\tau \Ccal} \rme^{\frac{\tau}{2} \Acal} \rme^{\frac{\tau}{2} \Bcal}u_n.
\end{equation}
In this double logarithmic diagram, the errors of a method lie on a straight line of slope $q$, where $q$ denotes the order of the method. Both, Lie and Strang splitting show their expected orders of convergence. Lie splitting has order one, as proved in Theorem~\ref{thm:convergence}, whereas Strang splitting converges with order two.

Note that the computationally most expensive task in the employed splitting 
approach is the advection step. Therefore, we order the steps in \eqref{eq:strang} in such a way that the advection equation is solved only once in each time step. In this way, Strang splitting provides much more accuracy without being significantly more expensive than Lie splitting.

Next, we compare the three different numerical realizations of the advection step, namely by local interpolation, by direct Fourier series evaluation (DFT) and by NFFT. In Table~\ref{tab:massconservation1d} we report the error in mass conservation and the required CPU time for various values of $N$. The number of time steps is fixed to $n=128$. The error in mass conservation is measured as the maximum deviation from the initial mass on the discrete level ($l^2$ in space and $l^\infty$ in time).

\begin{table}[!ht]
\centering
\renewcommand{\arraystretch}{1.2}
\begin{tabular}{|*{8}{c|}}
\hline
 & & \multicolumn{4}{c|}{interpolation} & \multicolumn{2}{c|}{Fourier}\\
\cline{3-8}
$N$ & & $p=2$ & $p=4$ & $p=6$ & $p=8$ & DFT & NFFT\\
\hline
128 & mass & 1.4e-01 & 1.8e-02 & 2.1e-03 & 2.8e-04 & 2.8e-15 & 8.6e-14 \\
\cline{2-8}
 & CPU & 0.13 & 0.12 & 0.12 & 0.12 & 0.10 & 0.16 \\
\hline
256 & mass & 9.4e-02 & 2.7e-03 & 7.2e-05 & 2.5e-06 & 2.0e-15 & 1.0e-14 \\
\cline{2-8}
 & CPU & 0.13 & 0.13 & 0.13 & 0.14 & 0.19 & 0.17 \\
\hline
512 & mass & 5.2e-02 & 2.9e-04 & 2.0e-06 & 1.8e-08 & 3.6e-15 & 1.7e-14 \\
\cline{2-8}
 & CPU & 0.16 & 0.19 & 0.17 & 0.16 & 0.27 & 0.19 \\
\hline
1024 & mass & 1.6e-02 & 1.8e-05 & 3.0e-08 & 9.6e-11 & 4.0e-15 & 5.5e-14 \\
\cline{2-8}
 & CPU & 0.22 & 0.23 & 0.23 & 0.24 & 0.56 & 0.23 \\
\hline
2048 & mass & 4.2e-03 & 1.1e-06 & 4.9e-10 & 3.8e-12 & 3.3e-15 & 1.3e-14 \\
\cline{2-8}
 & CPU & 0.36 & 0.37 & 0.37 & 0.37 & 1.42 & 0.33 \\
\hline
\end{tabular}
\caption{Error in mass conservation and CPU time (in seconds) for the
one-dimensional numerical example.}
\label{tab:massconservation1d}
\end{table}
Due to the compressive behavior of $S'_0(x)$, which acts as an initial velocity, the evolution develops caustics and the numerical solution requires a sufficiently large number $N$ of Fourier modes in order to reproduce accurate physical observables. While DFT and NFFT always preserve the mass almost up to machine precision, the polynomial methods become comparable only with the largest tested value of $N$ and at polynomial degree 7. For this degree, they are slightly more expensive than the NFFT approach.

The second numerical experiment is set in the two-dimensional domain
$[-5,5]^2$ with
\begin{equation*}
\begin{aligned}
A_1(x,y)&=-3\sin\left(\tfrac{2\pi(y+5)}{10}\right),\\
A_2(x,y)&=3\sin\left(\tfrac{2\pi(x+5)}{10}\right),\\
V(x,y)&= 20\cos\left(\tfrac{2\pi(x+5)}{10}\right)+ 20\cos\left(\tfrac{2\pi(y+5)}{10}\right)+40,
\end{aligned}
\end{equation*}
and initial value
\begin{equation*}
u_0(x,y)=\sqrt{\tfrac{\sqrt{10}}{\pi}}\exp\left(-\tfrac{\sqrt{10}}{2}\left((x-1)^2+y^2\right)\right).
\end{equation*}
The semi-classical parameter is chosen $\varepsilon=1$, the final time $T=50$ and the number of time steps $n=1000$. In Table~\ref{tab:massconservation2d} we compare the three methods that only differ in the treatment of the advection step. In particular, we compare the behavior of tensor interpolation at $4\times 4$ and $6\times 6$ points with direct Fourier series evaluation and NFFT with the default value $m=8$ and the smaller values $m=6$ and $m=4$, respectively.

\begin{table}[!ht]
\centering
\renewcommand{\arraystretch}{1.2}
\begin{tabular}{|*{8}{c|}}
\hline
& & \multicolumn{2}{c|}{} & \multicolumn{4}{c|}{Fourier}\\
\cline{5-8}
& & \multicolumn{2}{c|}{interpolation} & & \multicolumn{3}{c|}{NFFT}\\
\cline{3-4}\cline{6-8}
$N_1=N_2$ & & $p=4$ & $p=6$ & DFT & $m=8$ & $m=6$ & $m=4$\\
\hline
128 & mass & 1.0e-01 & 2.5e-03 & 9.9e-11 & 1.0e-10 & 2.4e-10 & 2.2e-07\\
\cline{2-8}
& CPU & 25.2 & 33.5 & 174.3 & 23.7 &  22.8 & 20.6\\
\hline
256 & mass & 6.9e-03  & 3.9e-05 & 1.3e-08 & 1.3e-08 & 2.0e-08  & 2.5e-02\\
\cline{2-8}
& CPU & 101.7 & 117.9 & 2254 & 99.6  & 85.6 & 87.7\\
\hline
512 & mass & 4.3e-04  &  6.2e-07 & * & 9.7e-11  & 2.5e-10 & 2.0e-07\\
\cline{2-8}
& CPU &  412.7 & 506.8 & * & 435.7 & 401.4 & 400.4\\
\hline
1024 & mass & 2.7e-05  &  9.6e-09 & * & 9.7e-11  & 2.5e-10 & 1.9e-07\\
\cline{2-8}
& CPU &  1796 & 2139 & * & 1948 & 1840 & 1709\\
\hline
\end{tabular}
\caption{Error in mass conservation and CPU time (in seconds) for the two-dimensional numerical example.}
\label{tab:massconservation2d}
\end{table}
We observe that, for this long-term simulation, the mass is always well conserved by the direct Fourier series evaluation and by NFFT with the default value $m=8$. On the other hand, if
$m$ is halved, there is a significant degradation, especially with $N_1=N_2=256$. The direct Fourier series evaluation is much more expensive than the other methods, being impracticable for $N_1=N_2\ge 512$. The interpolation methods roughly cost as much as the NFFT approach, but their mass preservation is by far worse.

The final numerical example is a three-dimensional variation of the previous one. In the domain $[-5,5]^3$, with $\varepsilon=1$, we chose
\begin{equation*}
\begin{aligned}
A_1(x,y,z)&=\sin\left(\tfrac{2\pi(y+5)}{10}\right)+ \sin\left(\tfrac{2\pi(z+5)}{10}\right)\\
A_2(x,y,z)&=\sin\left(\tfrac{2\pi(x+5)}{10}\right)+ \sin\left(\tfrac{2\pi(z+5)}{10}\right)\\
A_3(x,y,z)&=\sin\left(\tfrac{2\pi(x+5)}{10}\right)+ \sin\left(\tfrac{2\pi(y+5)}{10}\right)\\
V(x,y,z)&= 20\cos\left(\tfrac{2\pi(x+5)}{10}\right)+ 20\cos\left(\tfrac{2\pi(y+5)}{10}\right)+20\cos\left(\tfrac{2\pi(z+5)}{10}\right)+ 60,
\end{aligned}
\end{equation*}
and the initial value
\begin{equation*}
u_0(x,y,z)=\tfrac{2^{3/8}}{\pi^{3/2}}\exp\left(-\tfrac{\sqrt{2}}{2}\left((x-1)^2+y^2+z^2\right)\right).
\end{equation*}
\begin{table}[!ht]
\centering
\renewcommand{\arraystretch}{1.2}
\begin{tabular}{|c|c|c|c|}
\hline
 & & \multicolumn{2}{c|}{NFFT}\\
\cline{3-4}
$N_1=N_2=N_3$ & & \verb+PRE_PSI+ & \verb+PRE_FULL_PSI+\\
\hline
16 & mass & 6.1e-13 & 6.1e-13\\
\cline{2-4}
   & CPU & 5.6 & 6.5\\
\hline
32 & mass & 8.2e-14 & 8.2e-14\\
\cline{2-4}
   & CPU & 37.7 & 51.7\\
\hline
64 & mass & 7.1e-13 & *\\
\cline{2-4}
   & CPU & 396.5 & *\\
\hline
128 & mass & 7.9e-09 & *\\
\cline{2-4}
   & CPU & 2976 & *\\
\hline
\end{tabular}
\caption{Error in mass conservation and CPU time (in seconds) for the three-dimensional example.}
\label{tab:massconservation3d}
\end{table}%
With this example, we also tested the option \verb+PRE_FULL_PSI+ of NFFT (see~\cite{KKP09}). At the price of a full precomputation of the window functions, which requires a storage of $(2m+1)^3\prod_i N_i$ double precision numbers, this option should allow an overall faster execution. In Table~\ref{tab:massconservation3d} we display the error of mass conservation and the CPU time for simulations up to $T=5$ with 100 time steps. As expected, there is no difference in the mass conservation property between the two schemes. However, we never succeeded in getting the \verb+PRE_FULL_PSI+ version faster than the default one (named \verb+PRE_PSI+). For $N_1=N_2=N_3\ge64$, it was even not possible to store the precomputed values in the RAM (8 GB). Nevertheless, the default implementation of NFFT, which requires a storage of $3(2m+1)\prod_i N_i$ for the window
functions, works without any problem.

\section{Conclusions}
In this paper we considered the numerical solution of the linear Schr\"odinger equation with a vector potential. The structure of the problem suggested to use a splitting method involving three different parts, namely a multiplicative term coming from scalar potentials, the Laplacian, and the advective term due to the vector potential. After establishing convergence of Lie splitting for an abstract problem, we analysed the required assumptions in the specific case of the magnetic Schr\"odinger equation. For the advection step, the solution along the characteristic curves was approximated by a nonequispaced fast Fourier transform. It turned out to be as fast as local polynomial interpolation and as accurate as direct Fourier series evaluation in the mass conservation at discrete level. Therefore, it can be considered as a competitive tool in the solution of advection equations with the method of characteristics.

\bibliographystyle{plain}

\end{document}